\documentclass[12pt,reqno]{amsart}

\usepackage{amsmath}
\usepackage{amsthm}
\usepackage{amssymb}
\usepackage{mathrsfs}
\usepackage{enumitem}

\numberwithin{equation}{section}
\newtheorem{theo}{Theorem}[section]
\newtheorem{prop}[theo]{Proposition}

\newtheorem{claim}[theo]{Claim}
\theoremstyle{definition}
\newtheorem{defi}[theo]{Definition}
\newtheorem{rem}[theo]{Remark}

\numberwithin{equation}{section}
\newcommand{\calF}{\mathcal F}

\newcommand{\PP}{\mathbb P}
\renewcommand{\SS}{\mathbb S}
\newcommand{\gap}{\mathbf A}

\begin{document}

\title[trees and gaps]{Trees and gaps from a construction scheme}

\author{Fulgencio Lopez}
\address{Department of Mathematics\\ University of Toronto\\ Bahen Center 40 St. George St.\\ Toronto, Ontario M5S 2E4, Canada.}
\email{fulgencio.lopez@mail.utoronto.ca}
\author{Stevo Todorcevic}
\address{Department of Mathematics, University of Toronto, Bahen Center 40 St. George St. Toronto, Ontario M5S 2E4, Canada.}
\address{Institut de Math\'ematiques de Jussieu, UMR 7586, 2 pl. Jussieu, case 7012, 75251 Paris Cedex 05, France.}
\email{stevo@math.toronto.edu, stevo.todorcevic@imj-prg.fr}

%\thanks{\itshape Submitted to Proc. Am. Math. Soc.}

\begin{abstract}
We present simple constructions of trees and gaps using a general construction scheme that can be useful in constructing many other structures.  As a result,  we solve a natural problem
about Hausdorff gaps in the quotient algebra $\mathcal{P}(\omega)/{\rm Fin}$  found in the literature.  As it is well known Hausdorff gaps can sometimes be filled in  $\omega_1$-preserving forcing extensions. There are two natural conditions on Hausdorff gaps,
dubbed $S$ and $T$ in the literature,  that guarantee the existence of such forcing extensions.  In part, these conditions are motivated by analogies between fillable Hausdorff gaps and Souslin trees.  While the condition $S$ is  equivalent to the existence of $\omega_1$-preserving forcing extensions that fill the gap, we show here that its natural strengthening  $T$ is in fact  strictly stronger.
\end{abstract}

\maketitle

\section{Introduction}

Souslin trees are important set-theoretic objects that were first considered in connection with the Souslin Hypothesis that characterizes the unit interval as the unique ordered continuum satisfying the countable chain condition (see \cite{Kurepa}). They are also important tools in many other considerations in set theory. Similarly, Hausdorff's
$(\omega_1, \omega_1)$-gaps in the quotient algebra $\mathcal{P}(\omega)/{\rm Fin}$ are important set theoretic tools that naturally show up in a wide range considerations in set theory and related areas (see, for example,\cite{DalesWoodin}). It turns out that there are numerous analogies between $(\omega_1, \omega_1)$-gaps and Aronszajn trees, trees of height $\omega_1$ that have countable levels and branches (see, for example, \cite{LarsonTodorcevic}). Souslin trees are very specific kind of Aronszajn trees since they may admit uncountable branches in $\omega_1$-preserving forcing extensions of the set-theoretic universe. Analogously, as it is well known, some $(\omega_1, \omega_1)$-gaps may be filled in $\omega_1$-forcing extensions of the universe, so this sort of gaps are sometimes called {\it Souslin gaps}, or in short S-gaps. A Ramsey-theoretic analysis of S-gaps further strengthens this analogy and more importantly points out to a natural variation of this notion, a notion of a T-gap that we introduce below. We show that these seemingly similar notions are in fact different. For showing this we introduce a technique that should find many other applications.
In \cite{Todor} the second author introduced the concept of a construction scheme to build examples of compact spaces, convex sets and normed spaces which had previously required forcing constructions (see \cite{BGT} and \cite{LAT}). The existence of construction schemes is deduced in \cite{Todor} from Jensen's $\diamondsuit$-principle.  The following are our specific results where the notions of `capturing construction scheme', `S-gap',`fillable gap' and `T-gap' are defined in the following section.

\begin{theo}\label{suslin}
Assume there is a Construction Scheme that is 3-capturing. Then there is a Souslin tree.
\end{theo}

\begin{theo}\label{t.gap}
Assume there is a 3-capturing Construction Scheme. Then there is a $(\omega_1, \omega_1)$-gap that is a T-gap and so, in particular there is  $(\omega_1, \omega_1)$-gap
that can be filled in a forcing extension over a partially ordered set satisfying the countable chain condition.
\end{theo}

Every T-gap is a fillable gap but the converse need not be true. More precisely, we have the following result
\begin{theo}\label{model.gaps}
There is a model of set theory in which there is a fillable $(\omega_1, \omega_1)$-gap but with no T-gaps.
\end{theo}

\section{Preliminaries}
For bounded subsets $A,B\subset\omega_1$ we say that $A<B$ if for every $a\in A$ and $b\in B$ we have $a<b$.
We will work with a special kind of $\Delta$-systems.
\begin{defi}
For $\gamma\leq \omega_1,$ we say that a sequence $(s_\alpha)_{\alpha<\gamma}$ of finite subsets of $\omega_1$ is an increasing $\Delta$-system with root $s$ if for every $\alpha<\beta<\gamma$ we have
$s_\alpha\cap s_\beta=s$ and
$s<(s_\alpha\setminus s)<(s_\beta\setminus s)$.
\end{defi}

Recall that every uncountable family of finite subsets of $\omega_1$ contains an uncountable increasing $\Delta$-system as above.

%____________________________________Construction schemes_____________________________________%

\subsection{Construction schemes}
In this section, we introduce the notion of a construction scheme.
The key feature of this scheme is that it provides a family $\calF$ of finite subsets of $\omega_1$ which
allow us to perform recursive constructions by amalgamating {\it many isomorphic} structures
of lower rank. These amalgamations will determine the behaviour of uncountable substructures of the limit structure  via an appropriate property of {\it capturing} of the construction scheme.
For a more detailed analysis, see \cite{Todor}.

\begin{defi}
Let $(m_k)_{k<\omega}, (n_k)_{1\leq k<\omega}$ and $(r_k)_{1\leq k<\omega}$ be sequences of natural numbers such that
$m_0=1$, $m_{k-1}>r_k$ for all $k>0$, $n_k>k$ and for every $r<\omega$ there are infinitely many $k$'s with $r_k=r$. If for every $k>0$ we have
$$ m_k=n_k(m_{k-1}-r_k)+r_k$$
we say that $(m_k,n_k,r_k)_{k<\omega}$ forms a {\it type}.
\end{defi}

\begin{defi}\label{cons.sch}
Let $\calF$ be a family of finite subsets of $\omega_1$ such that
\begin{enumerate}
\item For every $A\subset\omega_1$ finite, there is $F\in\calF$ such that $A\subset F$.
\end{enumerate}
We say that $\calF$ is a \emph{construction scheme of type} $(m_k,n_k,r_k)_{k<\omega}$ if there are two mappings
$$ \rho:\calF\longrightarrow\omega\qquad R:\calF\longrightarrow[\omega_1]^{<\omega}$$
such that for every $F\in\calF$, with $\rho^F=k>0$ the following holds
\begin{enumerate}[resume]
\item $|F|=m_k$ and $|R(F)|=r_k$.
\item there are unique $F_i\in\calF\ (i<n_k)$ such that, $\rho^{F_i}=k-1$ and 
$$   F=\bigcup_{i<n_k} F_i $$
 
Furthermore $(F_i)_{i<n_k}$ forms an increasing $\Delta$-system with root $R(F),$ i.e.,
$$									R(F)< F_0\setminus R(F) <\ldots < F_{n_k-1}\setminus R(F) $$
\end{enumerate}
\end{defi}

We call $\rho^F$ the \emph{rank} of $F$ and the sequence $(F_i)_{i<n_k}$ of (3) the \emph{canonical decomposition} of $F$.

It is proved in \cite{Todor} that for any type $(m_k,n_k,r_k)_{k<\omega}$ there is a construction scheme  with that type.
To avoid confusion we will use $m_k, n_k$ and $r_k$ as above and we will omit reference to the type of a construction scheme.

For two $F,E\in\calF$ of the same rank there is a unique order-preserving bijection, we denote this map by $\varphi_{F,E}$.
In the particular case of $\varphi_{F_0,F_i}:F_0\rightarrow F_i$ we will simply write $\varphi_i$ when there is no confusion.
If $f$ is a function on $F_0$ then we can define the function $\varphi_i(f)$ in $F_i$ by $\gamma\mapsto f(\varphi_i^{-1}(\gamma))$.

We introduce now the concept of capturing
\begin{defi}
Let $\calF$ be a construction scheme. We say that $\calF$ is $n$-\emph{capturing} if for every uncountable $\Delta$-system
$(s_\xi)_{\xi<\omega_1}$ of finite subsets of $\omega_1$ with root $s$ there is a sub-$\Delta$-system $(s_{\xi_i})_{i<n}$ and $F\in\calF$ such that
\begin{gather*}
s\subset R(F)\\
s_0\setminus s\subset F_0\setminus R(F) \\
\varphi_i(s_{\xi_0}\setminus s)=s_{\xi_i}\setminus s\subset F_i\setminus R(F) \quad (i<n),
\end{gather*}
\end{defi}
where $   F=\bigcup_{i<n_k} F_i $ is the canonical decomposition of $F$ with $k=\rho^F>0.$

In \cite{Todor} it is shown that the existence of a Construction Scheme which is $k$-capturing for arbitrarily long $k<\omega$,
follows from $\diamondsuit$ and can be used to construct a large spectrum of different  examples of  mathematical structures  motivated by some previous forcing constructions (see, \cite{BGT} and  \cite{LAT}).
We will see below that only $3$-capturing is enough to construct some other interesting combinatorial objects.

%__________________________Gaps___________________________________%

\subsection{Gaps in $[\omega]^\omega$}
We recall the definition of gap in $[\omega]^\omega$ as well as some well known results.

For $a$ and $a'$, infinite subsets of $\omega$ we say $a\subseteq^* a'$ if $a\setminus a'$ is finite.

\begin{defi}
We say $(a_\alpha,b_\alpha)_{\alpha<\omega_1}$, with $a_\alpha,b_\alpha\subset\omega$ infinite,
is a \emph{pre-gap} if for every $\alpha<\beta<\omega_1$
\begin{enumerate}
\item $a_\alpha\cap b_\alpha=\emptyset$.
\item $a_\alpha\subseteq^* a_\beta$ and $b_\alpha\subseteq^* b_\beta$.
\item $a_\delta\cap b\gamma$ is finite for every $\delta,\gamma<\omega_1$.
\end{enumerate}

We say that $(a_\alpha,b_\alpha)_{\alpha<\omega_1}$ is a \emph{gap} if it is a pre-gap and
\begin{enumerate}[resume]
\item there is no infinite $c\subset\omega$ such that
\begin{enumerate}
                    \item $a_\alpha\subseteq^* c$ for every $\alpha<\omega_1$.
										\item $b_\alpha\cap c$ is finite for every $\alpha<\omega_1$.
\end{enumerate}
\end{enumerate}
\end{defi}

The existence of gaps is due to Hausdorff \cite{Haus}.
Recall the following Ramsey property of gaps
\begin{prop}\label{ramseygapcondition}
             A pre-gap $(a_\alpha,b_\alpha)_{\alpha<\omega_1}$ form an $(\omega_1,\omega_1)$-gap if and only if for every uncountable
						 $\Gamma\subset\omega_1$ there are $\alpha<\beta$ in $\Gamma$ such that $a_\alpha\cap b_\beta\neq\emptyset$.
\end{prop}

\begin{defi}
             We say a gap $(a_\alpha,b_\alpha)_{\alpha<\omega_1}$ is a S-gap if for every uncountable
						$\Gamma\subset\omega_1$ there are $\alpha<\beta$ in $\Gamma$ such that $(a_\alpha\cap b_\beta)\cup(a_\beta\cap b_\alpha)=\emptyset$.
\end{defi}

The existence of a S-gap is independent of ZFC. In \cite{ADow} a S-gap is constructed using $\diamondsuit$
and the next proposition implies that under MA$_{\aleph_1}$ all gaps are indestructible (see e.g. \cite{FarTo}).
\begin{prop}\label{suslingapcondition}
             The following are equivalent:
						\begin{enumerate}
						\item There is an $\omega_1$-preserving forcing notion that splits $(a_\alpha,b_\alpha)_{\alpha<\omega_1}$.
						\item The forcing notion defined by $p\in\PP=[\omega_1]^{<\omega}$ iff $a_\alpha\cap b_\beta=\emptyset$
						      for all $\alpha\neq\beta$ in $p$ ordered by extension has the ccc.
						\item For every uncountable $\Gamma\subset\omega_1$ there are $\alpha<\beta$ in $\Gamma$ such that
						$(a_\alpha\cap b_\beta)\cup(a_\beta\cap b_\alpha)=\emptyset$.
						\end{enumerate}
\end{prop}

In the literature, $(\omega_1, \omega_1)$-gaps with these properties are called `destructible gaps', `fillable gaps', `Souslin gaps' or `S-gaps.' This definition leads us to the following natural strengthening.
\begin{defi}\label{t.gap.def}
We say a gap $(a_\alpha,b_\alpha)_{\alpha<\omega_1}$ is a \emph{T-gap} if for every uncountable $\Gamma\subset\omega_1$ there are
$\alpha<\beta$ such that $a_\alpha\subseteq a_\beta$ and $b_\alpha\subseteq b_\beta$.
\end{defi}

We will show that it is consistent that there are S-gaps but no T-gaps.

We give the proofs of the propositions for the convenience of the reader.
\begin{proof}[Proof of Proposition~\ref{ramseygapcondition}]
Suppose $(a_\alpha,b_\alpha)_{\alpha<\omega_1}$ is not a gap and let $c\subset\omega$ witness this.
There is $n<\omega$ and uncountable $\Gamma\subset\omega_1$ such that $a_\alpha\setminus c\subset n$ and $b_\alpha\cap c\subset n$
for all $\alpha\in\Gamma$. We can also assume that there are $s,t\subset n$ such that for every $\alpha\in\Gamma$
$a_\alpha\cap n=s$ and $b_\alpha\cap n=t$. The condition $a_\alpha\cap b_\alpha=\emptyset$ implies that $s\cap t=\emptyset$.

For every $\alpha<\beta$ in $\Gamma$ we have
$$ a_\alpha\cap b_\beta =(a_\alpha\cap n)\cap (b_\beta\cap n)=s\cap t=\emptyset$$

Suppose now that there is $\Gamma\subset\omega_1$ uncountable such that $a_\alpha\cap b_\beta=\emptyset$ for every $\alpha<\beta$ in $\Gamma$. Define
$$ c=\bigcup_{\alpha\in\Gamma}a_\alpha $$
is clear that $a_\alpha\subset^* c$ for every $\alpha<\omega_1$.
We just have to check that $c\cap b_\gamma$ is finite for all $\gamma<\omega_1$.
Let $\gamma<\omega_1$.
Since $a_\alpha\cap b_\gamma$ is finite, if $c\cap b_\gamma$ is infinite there must be some $\delta\in\Gamma$ limit in $\Gamma$,
$\gamma<\delta$ such that
$$ \bigcup_{\alpha\in\Gamma\cap\delta}a_\alpha\cap b_\gamma\quad\mbox{is infinite} $$
but $b_\gamma\setminus b_\delta$ is finite and $\bigcup_{\alpha\in\Gamma\cap\delta}a_\alpha\cap b_\delta=\emptyset$, contradiction.

\end{proof}

\begin{proof}[Proof of Proposition~\ref{suslingapcondition}]
First we see $(3)\Rightarrow(2)\Rightarrow(1)$.
Let $\mathbb P$ be as in $(2)$.
Notice that $\mathbb P$ forces $(a_\alpha,b_\alpha)_{\alpha<\omega_1}$ to split by forcing $\Gamma\subset\omega_1$ without the property of Proposition~\ref{ramseygapcondition}.
We see that $\PP$ is ccc hence $\omega_1$-preserving.

Let $(p_\alpha)_{\alpha<\omega_1}$ in $\PP$.
There is uncountable $\Gamma\subset\omega_1$ such that:
\begin{enumerate}\renewcommand{\theenumi}{$(\roman{enumi})$}\renewcommand{\labelenumi}{\theenumi}
\item $(p_\alpha)_{\alpha\in\Gamma}$ forms a $\Delta$-system with $|p_\alpha|=k$.
\item If $p_\alpha=\{\delta^\alpha_{1}<\ldots<\delta^\alpha_k\}$ there is $n<\omega$ such that $a_{\delta^\alpha_i}\setminus n\subset a_{\delta^\alpha_k}$ and the same for $b_{\delta^\alpha_i}$.
\item There are $s_i,t_i\subset n$ for $i=1,\ldots,k$ such that $a_{\delta^\alpha_i}\cap n=s_i$ and $b_{\delta^\alpha_i}\cap n=t_i$.
\end{enumerate}
Note that $s_i\cap t_j=\emptyset$. Consider $\{\delta^\alpha_k\}_{\alpha\in\Gamma}$ by hypothesis
\begin{equation*}
\mbox{there are $\alpha<\beta$ in $\Gamma$
such that }(a_{\delta^\alpha_k}\cap b_{\delta^\beta_k})\cup(a_{\delta^\beta_k}\cap b_{\delta^\alpha_k})=\emptyset.
\end{equation*}

\noindent By $(iii)$ we have $(a_{\delta^\alpha_i}\cap b_{\delta^\beta_j})\cup(a_{\delta^\beta_j}\cap b_{\delta^\alpha_i})\cap n=\emptyset$, by $(ii)$ we have
$$(a_{\delta^\alpha_i}\cap b_{\delta^\beta_j})\cup(a_{\delta^\beta_j}\cap b_{\delta^\alpha_i})\setminus n\subset (a_{\delta^\alpha_k}\cap b_{\delta^\beta_k})\cup(a_{\delta^\beta_k}\cap b_{\delta^\alpha_k})=\emptyset$$
and $p_\alpha\cup p_\beta\in\PP$.

$(2)\Rightarrow(3)$ Let $\Gamma$ be an uncountable subset of $\omega_1$. Take $(p_\alpha=\{\alpha\})_{\alpha\in\Gamma}$
since $\PP$ has the ccc there is $\alpha<\beta$ in $\Gamma$ such that $p_\alpha\not\perp p_\beta$ but this implies
$(a_\alpha\cap b_\beta)\cup(a_\beta\cap b_\alpha)=\emptyset$ as we wanted.

$(1)\Rightarrow(2)$ Let $\mathbb Q$ be a forcing notion $\omega_1$-preserving that splits $(a_\alpha,b_\alpha)_{\alpha<\omega_1}$.
By the proof of Proposition~\ref{ramseygapcondition} for every $\dot{\Gamma}_0\subset\omega_1$ uncountable we can find $\dot{\Gamma}$ such that
\begin{align*}
\mathbb Q &\Vdash \dot{\Gamma}\subset\dot{\Gamma}_0 \mbox{ uncountable.} \\
\mathbb Q &\Vdash \mbox{``for every $\alpha<\beta$ in $\dot{\Gamma},\ (a_\alpha\cap b_\beta)\cup(a_\beta\cap b_\alpha)=\emptyset$.''}
\end{align*}
Applying $(2)\Leftrightarrow(3)$, which we already proved, $\mathbb Q\Vdash \mbox{``$\PP$  has the ccc''}$.
If $\PP$ has an uncountable antichain on the ground model
it has an uncountable antichain on $V^{\mathbb Q}$ because $\mathbb Q$ is $\omega_1$-preserving. Thus $\PP$ is ccc and we finish the proof.
\end{proof}

\section{Suslin tree}

In this section we will prove Theorem \ref{suslin}. Let $\calF$ be a construction scheme 3-capturing.
We will construct by recursion \textit{finite approximations} to an uncountable tree using the structure of $\calF$,
then the capturing property of $\calF$ will make this tree Suslin.

More precisely; for every $F\in\calF$ and every $\alpha\in F$ we construct
functions $f_\alpha^F,g_\alpha^F:F\rightarrow \{0,1\}$ such that
\begin{enumerate}
\item $f_\alpha^F\restriction\alpha=g_\alpha^F\restriction\alpha$
\item $f_\alpha^F(\alpha)=0$, $g_\alpha^F(\alpha)=1$.
\end{enumerate}
We want the functions to be isomorphic and coherent;
\begin{enumerate}[resume]
\item If $E,F\in\calF$ with $\rho^E=\rho^F$, $\alpha\in E$ and $\bar{\alpha}=\varphi_{E,F}(\alpha)$ then,
      $f_{\bar{\alpha}}^F=\varphi_{E,F}(f_\alpha^E)$.
\item If $E\subset F$, then for every $\alpha\in E$ we have
$$f_\alpha^E\subseteq f_\alpha^F\quad\mbox{and} \quad g_\alpha^E\subset g_\alpha^F$$
\end{enumerate}

We can now define $(h_\alpha: \alpha<\omega)$ such that

$$ h_\alpha\restriction F=f_\alpha^F\restriction(\alpha\cap F)=g_\alpha^F\restriction (\alpha\cap F)$$
for every $F\in\calF$ with $\alpha\in F$. Note that $h_\alpha:\alpha\rightarrow \{0,1\}$ and is well defined by (3) above and
property (1) of Definition \ref{cons.sch}. Now let
\begin{equation}\label{suslin.tree} \SS=(h_\alpha\restriction\delta:\delta\leq\alpha<\omega_1)\end{equation}

\begin{proof}[Proof of Theorem \ref{suslin}]

The functions $f_\alpha^F,g_\alpha^F (\alpha<\omega_1)$ will be defined by recursion based on the rank of $F\in\calF$.

For $\rho^F=0$ we have $F=\{\alpha\}$ and we let $f_\alpha^F(\alpha)=0$ and $g_\alpha^F(\alpha)=1$.

Let $F\in\calF$ with $\rho^F>0$, $R(F)=R$. Let $F=\bigcup_{i<n}F_i$
be its canonical decomposition and for all $i<n$, $f_\alpha^{F_i},g_\alpha^{F_i}$ are defined for all $\alpha\in F_i$
satisfying (1)--(4).

Let $\varphi_i:F_0\rightarrow F_i$be the increasing bijection between $F_0$ and $F_i$.

For $\alpha\in R$ let $f_\alpha^F=\bigcup_{i<n}\varphi_i(f_\alpha^{F_0})$ and $g_\alpha^F=\bigcup_{i<n}\varphi_i(g_\alpha^{F_0})$.

\vspace{.5cm}
For $\delta\in F_{2i}\setminus R$ and $\delta=\varphi_{2i}(\alpha)$ for some $\alpha\in F_0$ let
\begin{align*}
f_\delta^F=\bigcup_{j\leq 2i}\varphi_j(f_\alpha^{F_0})\cup\bigcup_{j>2i}\varphi_j(g_\alpha^{F_0}) \\
g_\delta^F=\bigcup_{j<2i}\varphi_j(f_\alpha^{F_0})\cup\bigcup_{j\geq 2i}\varphi_j(g_\alpha^{F_0})
\end{align*}

\vspace{.5cm}
For $\delta\in F_{2i+1}\setminus R$ and $\delta=\varphi_{2i+1}(\alpha)$ for some $\alpha\in F_0$ let
\begin{align*}
f_\delta^F=\bigcup_{j< 2i}\varphi_j(g_\alpha^{F_0})\cup\bigcup_{j\geq 2i}\varphi_j(f_\alpha^{F_0}) \\
g_\delta^F=\bigcup_{j\leq 2i}\varphi_j(g_\alpha^{F_0})\cup\bigcup_{j>2i}\varphi_j(f_\alpha^{F_0})
\end{align*}

It follows that for every $i<n$ and every $\alpha\in F_i$, $f_\delta^{F_i}\subset f_\delta^F$ and $g_\alpha^{F_i}\subset g_\alpha^{F}$
and (1)--(4) are preserved. This finish the recursion.

Let $\SS\subset 2^{<\omega_1}$ be as in \eqref{suslin.tree}.

\begin{claim}\label{suslintree}
            If $\calF$ is a 3-capturing construction scheme, then $\SS$ is a Suslin tree.
\end{claim}
\begin{proof}
It is clear that $\SS$ has height $\omega_1$ since for every $\alpha<\omega_1$, $h_\alpha\in \SS$. Next we see that $\SS$ has neither uncountable antichains not uncountable chains.

Let $W=(h_\alpha\restriction\delta_\alpha:\delta_\alpha\leq\alpha,\ \alpha\in\Gamma)\subset \SS$ with $\Gamma\subset\omega_1$ uncountable.

There are $\alpha<\beta$ in $\Gamma$ and $F\in\calF$ such that $F$ captures $\alpha$ and $\beta$. In particular $\beta=\varphi_1(\alpha)$
and then $h_\alpha\subset h_\beta$ which implies $(h_\alpha\restriction\delta_\alpha)\not\perp (h_\beta\restriction\delta_\beta)$.
This implies $\SS$ has no uncountable antichains.

In particular, the levels of $\SS$ are countable and we can find an uncountable $\Gamma_0\subset\Gamma$ such that for every $\alpha<\beta$ in $\Gamma_0$, $\alpha<\delta_\beta$.
Let $F\in\calF$,\  3-capture $\Gamma_0$. Thus there are $\alpha_0<\alpha_1<\alpha_2$ in $\Gamma_0$ captured by $F=\bigcup_{i<n_k}F_i$.
By the construction we have that $h_{\alpha_1}(\alpha_0)=g^{F_0}_{\alpha_0}(\alpha_0)=1$ and 
$h_{\alpha_2}(\alpha_0)=f^{F_0}_{\alpha_0}(\alpha_0)=0$ and since $\alpha_0<\delta_{\alpha_1},\delta_{\alpha_2}$ then $h_{\alpha_1}\perp h_{\alpha_2}$. Thus $\SS$ does not have uncountable chains. 
\end{proof}
We showed that $\SS$ is a Suslin tree which is what we wanted.
\end{proof}

\section{T-gap}

We construct a T-gap by recursively building finite approximations $(a^F_\alpha,b^F_\alpha:\alpha\in F)$ and $(N_k)_{k<\omega}$ such that

\begin{enumerate}
\item For $\rho^F=k$ and every $\alpha\in F$, $a^F_\alpha,b^F_\alpha\subset N_k$ and $a^F_\alpha\cap b^F_\alpha=\emptyset$.

\item For $E,F\in\calF$, $\rho^E=\rho^F$. If $\alpha\in E$ and $\bar{\alpha}=\varphi_{E,F}(\alpha)$ then
\begin{align*}
                a^{E}_\alpha &= a^{F}_{\bar{\alpha}} \\
								b^{E}_\alpha &= b^{F}_{\bar{\alpha}}
\end{align*}
\item If $E\subset F$ with $\rho^E=l <\rho^F$, then
\begin{enumerate}
                  \item For every $\alpha\in E$, $\quad a^F_\alpha\cap N_l=a^E_\alpha$ and $b^F_\alpha\cap N_l=b^E_\alpha$.
									\item For every $\alpha<\beta$ in  $E$, $\quad a^F_\alpha\setminus N_l\subset a^F_\beta$ and
									      $b^F_\alpha\setminus N_l\subset b^F_\beta$.
								  \item For every $\alpha,\beta\in E$, $\quad a^F_\alpha\cap b^F_\beta\subset N_l$.
\end{enumerate}

\end{enumerate}

\begin{proof}[Proof of Theorem \ref{t.gap}]
For $F=\{\alpha\}$ let $a_\alpha^{F}=\{0\}$ and $b_\alpha^{F}=\{1\}$ and $N_0=2$.

Suppose that $(a_\alpha^E,b_\alpha^E: \alpha\in E, \rho^E <k)$ satisfies (1)--(3).
For $F\in\calF$, $\rho^F=k$, if
$$ F=\bigcup_{i<n}F_i\mbox{ is the canonical decomposition of $F$.}$$
\vspace{.5cm}

\noindent For $\alpha\in R(F)$ let $a^F_\alpha=a^{F_0}_\alpha$ and $b^F_\alpha=B^{F_0}_\alpha$.
\vspace{.5cm}

\noindent For $\delta\in F_{2i}\setminus R(F)$ and $\delta=\varphi_{2i}(\alpha)$ for some $\alpha\in F_0$ let
\begin{align*}
                 a^F_\delta &= a^{F_0}_\alpha \cup \{N_{k-1}\} \\
								 b^F_\delta &= b^{F_0}_\alpha \cup \{N_{k-1}+1 \}
\end{align*}

\vspace{.5cm}

\noindent For $\delta\in F_{2i+1}\setminus R(F)$ and $\delta=\varphi_{2i+1}(\alpha)$ for some $\alpha\in F_0$ let
\begin{align*}
                 a^F_\delta &= a^{F_0}_\alpha \cup \{N_{k-1}+1\} \\
								 b^F_\delta &= b^{F_0}_\alpha \cup \{N_{k-1} \}
\end{align*}

\noindent Let $N_k=N_{k-1}+2$. It is clear that $a^F_\delta\cap b^F_\delta=\emptyset$ and (1)--(3) are satisfied.
This finish the recursion
\vspace{.5cm}

For $\alpha<\omega_1$ let
$$a_\alpha=\bigcup\{a^F_\alpha:\alpha\in F\in\calF\} \qquad b_\alpha=\bigcup\{b^F_\alpha:\alpha\in F\in\calF\} $$

Conditions (1)--(3) imply that $(a_\alpha,b_\alpha)_{\alpha<\omega_1}$ is a pre-gap.

We use Proposition~\ref{ramseygapcondition} and Definition~\ref{t.gap.def} to see that $(a_\alpha,b_\alpha)_{\alpha<\omega_1}$ is a T-gap.
Let $\Gamma\subset\omega_1$ uncountable. Since $\calF$ is 3-capturing there is
$F\in\calF$ of rank $k$ and $\alpha_0<\alpha_1<\alpha_2$ in $\Gamma$ captured by $F$ i.e, $\alpha_i\in F_i\setminus R(F)$  for $i<3$ and $\alpha_j=\varphi_j(\alpha_0)$ for $j=1,2$. By the construction of $a_{\alpha_i},b_{\alpha_i} (i<3)$ we have that
$a_{\alpha_i}\cap N_k= a_{\alpha_i}^F$ and $b_{\alpha_i}\cap N_k=a_{\alpha_i}^F$. This and (b) of (3) give
\begin{gather} \label{eq.gap.1}
                   a_{\alpha_0}\cap b_{\alpha_1}\neq\emptyset \\ \label{eq.gap.2}
									 a_{\alpha_0}\subset a_{\alpha_2}\quad\mbox{and}\quad b_{\alpha_0}\subset b_{\alpha_2}
\end{gather}
Equation \eqref{eq.gap.1} implies $(a_\alpha,b_\alpha)_{\alpha<\omega_1}$ is a gap and by \eqref{eq.gap.2}  it is a T-gap as we wanted to see.
\end{proof}

\section{T-gaps vs S-gaps}

We prove Theorem~\ref{model.gaps}.
\begin{theo}
            There is a model in which there is an S-gap but which does not have any T-gaps.
\end{theo}
\begin{proof}
We start with a ground model in which GCH holds and has an S-gap.

Let $(a_\alpha,b_\alpha)_{\alpha<\omega_1}$ be a gap
with the property that $a_\beta\not\subset a_\alpha$ for any $\alpha<\beta<\omega_1$. It is clear that every gap is equivalent to a gap with this property. Let $\gap=(a_\alpha)_{\alpha<\omega_1}$ and consider the following forcing notion
$$
\mathbb P_\gap=\{p\in[\gap]^{<\omega}: (\forall x\neq y\in p)\ x\not\subset y \mbox{ and }y\not\subset x\}
$$
ordered by reversed inclusion.

\begin{claim}
$\PP_\gap$ is ccc.
\end{claim}
\begin{proof}
Let $(p_\alpha)_{\alpha<\omega_1}$. Applying the $\Delta$-system Lemma we can assume that the $p_\alpha$'s are a disjoint
with $|p_\alpha|=n$ and $p_\alpha=(x_{\alpha,i})_{i<n}$ for every $\alpha<\omega_1$ where we preserved the natural order in $\gap$.
This implies that $x_{\beta,j}\not\subset x_{\alpha,i}$ for $\alpha<\beta$ and $i,j<n$.

Let $M$ be a countable elementary submodel of $H_{\mathfrak c^+}$ and $\gamma=\omega_1\cap M$.

Take $\beta>\gamma$ and fix $k<\omega$ such that
\begin{equation}\label{ccc.eq.1}
x_{\beta,i}\cap k\not\subset x_{\gamma,i}\quad\forall i<n.
\end{equation}

Consider $\Gamma=\{\alpha<\omega_1:x_{\alpha,i}\cap k= x_{\beta,i}\cap k\quad\forall i<n\}$, then $\Gamma\in M$ and $\beta\in\Gamma$.
Therefore $\Gamma$ is uncountable. Take $\alpha\in M\cap\Gamma$, by \eqref{ccc.eq.1}
$$ x_{\alpha,i}\not\subset x_{\gamma,i}\quad\forall i<n$$
and $p_\alpha\cup p_\gamma\in\PP_\gap$ witness $p_\alpha\not\perp p_\gamma$.
\end{proof}

We will force a model where $\mbox{MA}_{\omega_1}$ holds for a forcing of the form $\PP_\gap$. First, fix a bijective mapping $\pi:\omega_2\rightarrow\omega_2\times\omega_2$ where $\pi(\alpha)=(\beta,\gamma)$ with $\beta\leq\alpha$. This is the usual book keeping mapping. Suppose we have $\PP_\lambda=\langle\PP_\alpha,\dot{\mathbb Q}_\alpha:\alpha<\lambda \rangle$ a finite support iteration with
$$\PP_\alpha\Vdash ``\dot{\mathbb Q}_\alpha=\PP_{\dot{\gap}}\mbox{ if $\dot{\gap}$ is a gap}".$$
for some $\dot{\gap}\in V^{\PP_\alpha}$. Then, in $V^{\PP_\lambda}$ there are $\aleph_2$ many names for gaps (by GCH), and we can fix a well-ordering of them.
If $\pi(\lambda)=(\beta,\gamma)$, let $\dot{\gap}$ be the $\gamma^{th}$ name for a gap in $V^{\PP_\beta}$. If $\dot{\gap}$ is a gap in $V^{\PP_\lambda}$ then let $\dot{\mathbb Q}_\lambda=\PP_{\dot{\gap}}$.

\begin{claim}
               The finite support iteration $\PP_{\omega_2}$ is ccc and forces {\upshape MA}$_{\omega_1}$ for orderings of the form $\PP_\gap$.
\end{claim}
\begin{proof}
Let $\gap$ and $\vec{\mathcal D}=(D_\alpha:\alpha<\omega_1)$ be a gap  and a collection of dense sets of $\PP_{\dot{\gap}}$ in $V[G_{\omega_2}]$ respectively. Then, there is $\lambda<\omega_2$ such that both $\gap$ and $\vec{\mathcal D}$ are in $V[G_\lambda]$. Since $\gap$ is a gap in $V[G_{\omega_2}]$ then is a gap in $V[G_\lambda]$ and there is $\xi\geq\lambda$ such that $\pi(\xi)=(\lambda,\gamma)$ and the $\gamma^{th}$ name in $V^{\PP_\lambda}$ is a name for $\gap$. It follows that there is a $\vec{\mathcal D}$-generic filter in $V[G_{\xi+1}]\subset V[G_{\omega_2}]$ and the proof is finished.  
\end{proof}

This applied to a gap $(a_\alpha,b_\alpha)_{\alpha<\omega_1}$ forces $\Gamma\subset\omega_1$ uncountable without the property in
Definition~\ref{t.gap}. This shows that there are no T-gaps. Thus, the proof  is finished once we show the following.

\begin{claim}
              Forcing with $\mathbb P_\gap$ preserves S-gaps.
\end{claim}
\begin{proof}
Suppose that one $\mathbb P_\gap$ kills an S-gap $(a_\alpha,b_\alpha)_{\alpha<\omega_1}$.

Then $\PP_\gap$ forces $\dot{\Gamma}\subset\omega_1$ uncountable without property (3) of Proposition~\ref{suslingapcondition} i.e,
for every $\alpha<\beta$
$$ \PP_\gap\Vdash\alpha,\beta\in\dot{\Gamma}\Rightarrow (a_\alpha\cap b_\beta)\cup(a_\beta\cap b_\alpha)\neq\emptyset$$

Since $\dot{\Gamma}$ is uncountable we can find (in the ground model) $\Gamma\subset\omega_1$ uncountable and 
$(p_\alpha:\alpha\in\Gamma)\subset\PP_\gap$ such that
$$ p_\alpha\Vdash\alpha\in\dot{\Gamma}$$

In particular, we have
\begin{equation}\label{Claim.eq.1}
\forall\alpha<\beta\in \Gamma\quad\Bigl( (a_\alpha\cap b_\beta)\cup(a_\beta\cap b_\alpha)=\emptyset
\Longrightarrow p_\alpha\cup p_\beta\notin\mathbb P_\gap \Bigr)
\end{equation}

We may assume that the $p_\gamma$'s are disjoint and that they all have some fixed size $n$
and $p_\alpha=(x_{\alpha,i})_{i<n}$ preserves the natural order in $\gap$.

Choose a countable elementary sub-model $M$ of $H_{c^+}$ containing all these objects and let $\gamma=\min (\Gamma\setminus M)$.

Since $(a_\alpha,b_\alpha)_{\alpha<\omega_1}$ is an S-gap, the elementarity of $M$ gives us the existence of a $\beta\in\Gamma$
above $\gamma$ such that
\begin{equation}\label{Claim.eq.3}
a_\beta\cap b_\gamma =\emptyset \mbox{ and } a_\gamma\cap b_\beta=\emptyset
\end{equation}

Choose $k<\omega$ such that

\begin{gather}\label{Claim.eq.4}
a_\gamma\setminus k\subseteq a_\beta \mbox{ and } b_\gamma\setminus k\subseteq b_\beta \\
\label{Claim.eq.5} \forall x\in p_\gamma\quad\forall y\in p_\beta\quad y\cap k\not\subseteq x\cap k
\end{gather}
Let $s=a_\beta\cap k$, $t=b_\beta\cap k$ and
$$\Gamma_0=\{\alpha\in\Gamma: a_\alpha\cap k=s\quad b_\alpha\cap k=t\quad x_{\alpha,i}\cap k=x_{\beta,i}\cap k (i<n)\}$$
Then $\Gamma_0\in M$ and $\beta\in\Gamma_0\setminus M$ so $\Gamma_0$ is an uncountable subset of $\Gamma$.
Since $(a_\alpha,b_\alpha)_{\alpha<\omega_1}$ is a S-gap there must exist $\alpha\in\Gamma_0\cap M$ such that
\begin{equation}\label{Claim.eq.6}
a_\alpha\cap b_\beta =\emptyset \qquad a_\beta\cap b_\alpha=\emptyset
\end{equation}

Combining equations \eqref{Claim.eq.3},\eqref{Claim.eq.4} and \eqref{Claim.eq.6} we obtain that
\begin{equation}\label{Claim.eq.7}
a_\alpha\cap b_\gamma=\emptyset \qquad a_\gamma\cap b_\alpha=\emptyset
\end{equation}
Form the fact that $\alpha\in\Gamma_0$ and by \eqref{Claim.eq.5} we conclude that
\begin{equation}\label{Claim.eq.8}
\forall x\in p_\gamma\quad\forall y\in p_\alpha\quad y\not\subseteq x
\end{equation}
Thus $p_\alpha\cup p_\gamma\in\mathbb P_W$, contradicting \eqref{Claim.eq.1}.
\end{proof}

The previous claim also implies that if $\PP_\alpha$ preserves S-gaps, then so does $\PP_{\alpha+1}$.
Suppose now $\PP_\alpha$ preserves S-gaps for every $\alpha<\lambda\leq\omega_2$, with $\lambda$ limit.
If $\PP_\lambda$ kills an S-gap, applying the $\Delta$-system lemma
(or a counting argument in case $\alpha$ has countable cofinality) we find $\eta<\lambda$ such that 
$\PP_\eta$ kills an S-gap. Contradiction, thus $\PP_\lambda$ also preserves S-gaps.

This shows that $V[G_{\omega_2}]$ contains an S-gap, since $V$ does and $\PP_{\omega_2}$ preserves it,
and there are no T-gaps in $V[G_{\omega_2}]$ which finish the proof.
\end{proof}

\begin{rem}
The method above answers a particular case of Problem 59 of \cite{BC}. Particularly, it produces a
model of with no Suslin towers but destructible gaps.
\end{rem}


\begin{thebibliography}{10}
%\normalsize
\parskip=3mm

\bibitem{BGT}{
M. Bell, J. Ginsburg, and S. Todor\v{c}evi\'{c}, \emph{Countable spread of $exp\ Y$ and $\lambda Y$}, Top. Appl. \textbf{14} (1982) 1--12
}

\bibitem{BC}{
P. Borodulin-Nadzieja and D. Chodounsk\'{y}, \emph{Hausdorff gaps and towers in $\mathcal{P}(\omega)/Fin$}, Fund. Math. \textbf{229} (2015) 197--229
}

\bibitem{DalesWoodin}{
H.G. Dales and W.H.  Woodin, \emph{An introduction to independence for analysts.} London Mathematical Society Lecture Note Series, 115. Cambridge University Press, Cambridge, 1987. xiv+241 pp.
}

\bibitem{ADow}{
A. Dow, \emph{More set-theory for topologist}, Top. Appl. \textbf{64} (1995) 243--300
}



\bibitem{Haus}{
F. Hausdorff, \emph{Summen von $\aleph_1$-mengen}, Fund. Math. \textbf{26} (1936) 241--255.
}

\bibitem{Kurepa}{
G. Kurepa \emph{Ensembles ordonn\'{e}s et ramifi\'{e}s,} Publ. Math. de l'Univ. Belgrade \textbf{4} (1935), 1--138.
}

\bibitem{LarsonTodorcevic}{
 P Larson and S. Todorcevic, \emph{Chain conditions in maximal models.} Fund. Math. \textbf{168} (2001), no. 1, 77--104.
}

\bibitem{LAT}{
J. L\'{o}pez-Abad and S. Todor\v{c}evi\'{c}. \emph{Generic Banach spaces and generic simplexes}, J. Funct. Anal. \textbf{261} (2011), 300--386.
}

\bibitem{Todor}{
S. Todor\v{c}evi\'{c}, \emph{A construction scheme for non-separable structures}, preprint 2014.
}

\bibitem{FarTo}{
S. Todorchevich and I. Farah, \emph{Some applications of the method of forcing}, Yenisei series in pure and applied mathematics, Moscow, 1995.
}


\end{thebibliography}
\end{document}